\RequirePackage{rotating}
\documentclass[12pt]{article}

\usepackage[margin=1in]{geometry}

\usepackage{amsmath}
\usepackage{amssymb}
\usepackage{amsthm}
\usepackage{url}
\usepackage{MnSymbol}

\usepackage{rotating}
\usepackage{mathdots}

\usepackage{enumitem}

\usepackage{hyperref}
\usepackage{verbatim}

\usepackage[all,cmtip]{xy}

\newtheorem{theorem}{Theorem}[section]
\newtheorem{lemma}[theorem]{Lemma}

\newtheorem{corollary}[theorem]{Corollary}
\newtheorem{proposition}[theorem]{Proposition}

\theoremstyle{definition}
\newtheorem{definition}[theorem]{Definition}

\newtheorem{question}[theorem]{Question}

\theoremstyle{remark}
\newtheorem{remark}[theorem]{Remark}

\newcommand{\mc}[1]{\mathcal{#1}}
\newcommand{\mf}[1]{\mathfrak{#1}}

\newcommand{\concat}[0]{\textrm{\^{}}}
\newcommand{\la}{\langle}
\newcommand{\ra}{\rangle}

\newcommand{\comp}{\mathtt{c}}

\newcommand{\Ahat}{\hat{\mc{A}}}
\newcommand{\Bhat}{\hat{\mc{B}}}
\newcommand{\Atilde}{\tilde{\mc{A}}}
\newcommand{\Btilde}{\tilde{\mc{B}}}

\DeclareMathOperator{\Dom}{Dom}

\DeclareMathOperator{\Th}{Th}
\DeclareMathOperator{\Iso}{Iso}
\DeclareMathOperator{\id}{id}

\setenumerate[1]{label={(\arabic*)}} 

\begin{document}

\title{Finitely Generated Groups Are Universal}

\author{Matthew Harrison-Trainor \and Meng-Che Ho}

\maketitle

\abstract{Universality has been an important concept in computable structure theory. A class $\mc{C}$ of structures is universal if, informally, for any structure, of any kind, there is a structure in $\mc{C}$ with the same computability-theoretic properties as the given structure. Many classes such as graphs, groups, and fields are known to be universal.

This paper is about the class of finitely generated groups. Because finitely generated structures are relatively simple, the class of finitely generated groups has no hope of being universal. We show that finitely generated groups are as universal as possible, given that they are finitely generated: for every finitely generated structure, there is a finitely generated group which has the same computability-theoretic properties. The same is not true for finitely generated fields. We apply the results of this investigation to quasi Scott sentences.}

\section{Introduction}

Whenever we have a structure with interesting computability-theoretic properties, it is natural to ask whether such examples can be found within particular classes. While one could try to adapt the proof within the new class, it is often simpler to try and code the original structure into a structure in the given class. It has long been known that for certain classes, such as graphs, this is always possible. Hirschfeldt, Khoussainov, Shore, and Slinko \cite{HKSS} proved that classes of graphs, partial orderings, lattices, integral domains, and 2-step nilpotent groups are ``complete with respect to degree spectra of nontrivial structures, effective dimensions, expansion by constants, and degree spectra of relations''. What this means is that for every structure $\mc{A}$, there is a structure in each of these classes which has the same degree spectrum, effective dimension, etc.\ as $\mc{A}$.

This list of properties is not a complete list of all possible computability-theoretic properties, but the method of proof is sufficiently general that any reasonable computability theoretic property can be added to the list. These classes of structures have been informally called universal.

Recently, Miller, Poonen, Schoutens, and Shlapentokh \cite{MPSS} added fields to the list of universal classes. In that paper, a category-theoretic language was used: it was shown that, for any structure $\mc{A}$, there is a $F$ and a faithful functor from copies of $F$ to copies of $\mc{A}$. It was shown that this implies that $F$ and $\mc{A}$ share various computability-theoretic properties. Such a functor can be found for each of the classes from \cite{HKSS}.

Around the same time, Montalb\'an \cite{MonFixed} introduced effective interpretations, which are interpretations as in model theory but using computable $\Sigma^{\comp}_1$ formulas, and with the domain of the interpreted structure allowing tuples of any finite size. Montalb\'an showed that two structures which are bi-interpretable share many of the same computability-theoretic properties. The results from \cite{HKSS} can also be phrased in terms of bi-interpretations.

Thus we have two different possible definitions of universality, one coming from the category-theoretic language, and one coming from the language of effective interpretations. In \cite{HTMelnikovMillerMontalban}, the first author together with Melnikov, Miller, and Montalb\'an showed that the two notions of universality are actually equivalent. 

This paper follows up on some observations made during the authors' previous work in \cite{HTHo}. There, the authors answered the question of whether every finitely generated group has d-$\Sigma^0_2$ Scott sentence by constructing a finitely generated group with no d-$\Sigma^0_2$ Scott sentence. The strategy was to first figure out how to build a finitely generated structure with no d-$\Sigma^0_2$ Scott sentence, and then to code that structure into a finitely generated group. We next wanted to work on several questions about quasi Scott sentences of finitely generated groups, and found ourselves using the same approach. While the class of finitely generated groups has no hope of being universal, we thought that we should try to use the ideas of universality. Our main result is:
\begin{theorem}\label{thm:main}
The class of finitely generated groups is universal among finitely generated structures, after naming finitely many constants.
\end{theorem}
\noindent What this means is that each finitely generated structure can be coded into a finitely generated group in a way that maintains its computability-theoretic properties. (Many, but not all, computability-theoretic properties are invariant under naming finitely many constants.) The statement of this theorem will be formally defined in Section \ref{sec:universal} and proved in Section \ref{sec:groups}. See Section \ref{sec:naming-constants} for an explanation of why we must name finitely many constants In Section \ref{sec:fields} we show that certain classes of structures, such as finitely generated fields, are not universal among finitely generated structures.

A quasi Scott sentence for a finitely generated structure $\mc{A}$ is an $\mc{L}_{\omega_1 \omega}$ sentence which describes $\mc{A}$ uniquely among finitely generated structures; thus there may be other structures, which are not finitely generated, which also satisfy $\varphi$. In Section \ref{sec:quasi}, we begin by proving some general results about quasi Scott sentences. We use the universality of finitely generated groups to transfer these results to that class.

\section{Background on Functors, Interpretations, and Universality}\label{sec:universal}

\subsection{Infinitary Sentences}

The infinitary logic $\mc{L}_{\omega_1 \omega}$ is the logic which allows countably infinite conjunctions and disjunctions but only finite quantification. Each formula has only finitely many free variables. If the conjunctions and disjunctions of a formula $\varphi$ are all over computable sets of indices for formulas, then we say that $\varphi$ is computable. 

We use the following recursive definition to define the hierarchy of complexity of $\mc{L}_{\omega_1\omega}$ formulas:
\begin{itemize}
\item An $\mc{L}_{\omega_1 \omega}$ formula is both $\Sigma_0^0$ and $\Pi_0^0$ if it is quantifier free and does not contain any infinite disjunction or conjunction. 
\item An $\mc{L}_{\omega_1 \omega}$ formula is $\Sigma_\alpha^0$ if it is a countable disjunction of formulas of the form $\exists \bar{x} \phi$ where each $\phi$ is $\Pi_\beta^0$ for some $\beta < \alpha$, and there are only finitely many free variables among all the formulas $\exists \bar{x} \phi$.
\item An $\mc{L}_{\omega_1 \omega}$ formula is $\Pi_\alpha^0$ if it is a countable conjunction of formulas of the form $\forall \bar{x} \phi$ where each $\phi$ is $\Sigma_\beta^0$ for some $\beta < \alpha$, and there are only finitely many free variables among all the formulas $\forall \bar{x} \phi$.
\end{itemize}
We say a formula is $\text{d-}\Sigma_\alpha^0$ if it is a conjunction of a $\Sigma_\alpha^0$ formula and a $\Pi_\alpha^0$ formula. The following diagram illustrates this hierarchy, from the simplest formulas on the left to the more complicated formulas on the right:

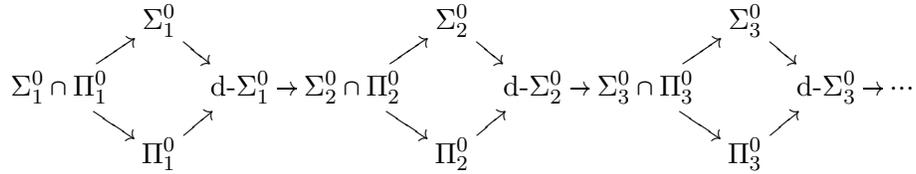
\[
\xymatrix@=7pt{
                                                      & \Sigma^0_1 \ar[dr]&                                      &                                                      &\Sigma_2^0\ar[dr]&&& \Sigma^0_3 \ar[dr] & & \\
\Sigma^0_1\cap\Pi^0_1 \ar[ur]\ar[dr] &                            &\text{d-}\Sigma^0_1 \ar[r]&\Sigma^0_2\cap\Pi^0_2 \ar[ur]\ar[dr]&&\text{d-}\Sigma^0_2 \ar[r]&\Sigma^0_{3}\cap\Pi^0_{3} \ar[ur] \ar[dr] & & \text{d-}\Sigma^0_3 \ar[r] & \cdots\\
                                                      & \Pi^0_1 \ar[ur]      &                                      &                                                     &\Pi_2^0\ar[ur]&&& \Pi^0_3    \ar[ur] & &
}
\]

\subsection{Effective Interpretations}

An interpretation is a way of defining one structure inside of another structure. In the interpretations which are traditional in model theory, the domain consists of tuples all of the same arity, and the definitions are in elementary first-order logic. See, for example, \cite[Definition 1.3.9]{Marker02}. Here, we want to use a different notion, introduced in \cite[Definition 1.7]{MonFixed}, where we use tuples of arbitrary lengths, and our definitions are computable sentences in the infinitary logic $\mc{L}_{\omega_1 \omega}$.

\begin{definition}\label{def: eff int}
We say that a structure $\mc{A} = (A; P_0^\mc{A},P_1^\mc{A},...)$ (where $P_i^\mc{A}\subseteq A^{a(i)}$) is {\em effectively interpretable} in $\mc{B}$ if there exist a $\Delta^{\comp}_1$-definable (in the language of $\mc{B}$, without parameters) sequence of relations $(\Dom_{\mc{A}}^{\mc{B}}, \sim, R_0, R_1,...)$ such that
\begin{enumerate}
	\item $\Dom^{\mc{B}}_{\mc{A}}\subseteq \mc{B}^{<\omega}$,
	\item $\sim$  is an equivalence relation on $\Dom^{\mc{B}}_{\mc{A}}$, 
	\item $R_i\subseteq (B^{<\omega})^{a(i)}$ is closed under $\sim$ within $\Dom_\mc{A}^\mc{B}$,
\end{enumerate}
and there exists a function $f^\mc{B}_\mc{A}\colon \Dom^{\mc{B}}_{\mc{A}} \to \mc{A}$ which induces an isomorphism: 
\[
(\Dom^{\mc{B}}_{\mc{A}}/\sim; R_0/\sim,R_1/\sim,...) \cong (A; P_0^\mc{A},P_1^\mc{A},...),
\]
where $R_i/ \sim$ stands for the $\sim$-collapse of $R_i$.
\end{definition}

There is also a notion of bi-interpretation, where not only is one structure interpretable in the other, and vice versa, but the two interpretations compose in a nice way.

\begin{definition}
\label{defn:eff-biinterpretable}
Two structures $\mc{A}$ and $\mc{B}$ are {\em effectively bi-interpretable} if there are effective
interpretations of each structure in the other as in Definition \ref{def: eff int} such that the compositions
\[
f^\mc{A}_\mc{B} \circ \tilde{f}^\mc{B}_\mc{A}\colon \Dom_\mc{B}^{(\Dom_\mc{A}^\mc{B})} \to \mc{B} 
\quad\mbox{ and }\quad
f^\mc{B}_\mc{A} \circ \tilde{f}^\mc{A}_\mc{B}\colon \Dom_\mc{A}^{(\Dom_\mc{B}^\mc{A})} \to \mc{A} 
\]
are $\Delta^{\comp}_1$-definable in $\mc{B}$ and $\mc{A}$ respectively.
(Here $\Dom_\mc{B}^{(\Dom_\mc{A}^\mc{B})}\subseteq (\Dom_\mc{A}^\mc{B})^{<\omega}$, and $\tilde{f}^\mc{B}_\mc{A}\colon (\Dom_\mc{A}^\mc{B})^{<\omega}\to \mc{A}^{<\omega}$ is the obvious extension of $f^\mc{B}_\mc{A}\colon \Dom_\mc{A}^\mc{B}\to \mc{A}$ mapping $\Dom_\mc{B}^{(\Dom_\mc{A}^\mc{B})}$ to $\Dom_\mc{B}^\mc{A}$.)
\end{definition}

Two structures which are bi-interpretable are essentially the same from the point of view of computability theory. In \cite[Lemma 5.3]{MonICM} it is shown that if $\mc{A}$ and $\mc{B}$ are effectively bi-interpretable then: 
they have the same degree spectrum;
they have the same computable dimension;
they have the same Scott rank;
their index sets are Turing equivalent (assuming the structures are infinite);
$\mc{A}$ is computably categorical if and only if $\mc{B}$ is;
$\mc{A}$ is rigid if and only if $\mc{B}$ is;
$\mc{A}$ has the  c.e.\ extendability condition if and only if $\mc{B}$ does;
for every $R\subseteq\mc{A}^{<\omega}$, there is a $Q\subseteq\mc{B}^{<\omega}$ which has the same relational degree spectrum, and vice-versa;
and the jumps of $\mc{A}$ and $\mc{B}$ are effectively bi-interpretable too.

To talk about universality, we want a notion of interpretability between classes of structures.

\begin{definition}[\cite{MonICM}]
Say that a class $\mf{C}$ is \emph{reducible via effective bi-interpretability} to a
class $\mf{D}$ if for every $\mc{C} \in \mf{C}$ there is a $\mc{D} \in \mf{D}$ such that $\mc{C}$ and $\mc{D}$ are effectively bi-interpretable and furthermore the formulas defining the bi-interpretation do not depend on the choice of $\mc{C}$ and $\mc{D}$.
\end{definition}

\subsection{Computable Functors}

We write $\Iso(\mc{A})$ for the isomorphism class of a countably infinite structure $\mc{A}$:
\[ \Iso(\mc{A}) = \{\Ahat~:~\Ahat\cong\mc{A} \text{ and } \text{dom}(\Ahat)=\omega\}. \]
We will regard $\Iso(\mc{A})$ as a category, with the copies of the structures as its objects and the isomorphisms among them as its morphisms.

\begin{definition}\label{def:functors}
By a {\em functor from $\mc{A}$ to $\mc{B}$} we mean a functor from $\Iso(\mc{A})$ to  $\Iso(\mc{B})$, that is, a map $F$ that assigns to each copy $\Ahat$ in $\Iso(\mc{A})$ a structure $F(\Ahat)$ in $\Iso(\mc{B})$, and assigns to each morphism $f\colon \Ahat \to \Atilde$ in $\Iso(\mc{A})$ a morphism $F(f) \colon F(\Ahat) \to F(\Atilde)$ in $\Iso(\mc{B})$ so that the  two properties below hold:
\begin{itemize}
	\item $F(\id_{\Ahat}) = \id_{F(\Ahat)}$ for every $\Ahat \in \Iso(\mc{A})$, and
	\item $F(f \circ g) = F(f) \circ F(g)$ for all morphisms $f,g$ in $\Iso(\mc{A})$.
\end{itemize}

A functor $F\colon \Iso(\mc{A}) \rightarrow \Iso(\mc{B})$ is {\em computable} if
there exist two computable operators $\Phi$ and $\Phi_*$ such that
\begin{itemize}
\item for every $\Ahat \in \Iso(\mc{A})$, $\Phi^{D(\Ahat)}$ is the atomic diagram of $F(\Ahat) \in \Iso(\mc{B})$;
\item for every morphism $f:\Ahat\to\Atilde$ in $\Iso(\mc{A})$,
$\Phi_*^{D(\Ahat)\oplus f\oplus D(\Atilde)}  = F(f). $
\end{itemize}
Here, $D(\Ahat)$ denotes the atomic diagram of $\Ahat$.
\end{definition}

Given an effective interpretation of a structure $\mc{A}$ inside of a structure $\mc{B}$, we get an induced computable functor from $\Iso(\mc{B})$ to $\Iso(\mc{A})$ by mapping a copy of $\mc{B}$ to the copy of $\mc{A}$ which is interpreted inside of it. One of the main results of \cite{HTMelnikovMillerMontalban} is that this reverses, i.e.\ each computable functor is induced by an effective interpretation in this way.

The next definition is the notion of computable isomorphism between computable functors.

\begin{definition}\label{def:effeq}
A functor $F\colon \Iso(\mc{B}) \rightarrow \Iso(\mc{A})$ is {\em effectively naturally isomorphic} (or just {\em effectively isomorphic}) to a functor $G\colon \Iso(\mc{B}) \rightarrow \Iso(\mc{A})$   if there is a computable Turing functional $\Lambda$ such that for every $\Btilde \in \Iso(\mc{B})$, $\Lambda^{\Btilde}$ is an isomorphism from $F(\Btilde)$ to $G(\Btilde)$, and the following diagram commutes for every $\Btilde, \Bhat \in \Iso(\mc{B})$ and every morphism $h\colon \Btilde \to \Bhat$:
\[
\xymatrix{
F(\Btilde)\ar[d]_{F(h)}\ar[r]^{\Lambda^{\Btilde}} &      G(\Btilde)\ar[d]^{G(h)}   \\
F(\Bhat)\ar[r]_{\Lambda^{\Bhat}}    & G(\Bhat)
}\]
\end{definition}

Given a class $\mf{C}$ of countable structures, we can view $\mf{C}$ as a category. The objects are the presentations, with domain $\omega$, of the structures in $\mf{C}$ and the morphisms are the isomorphisms between these presentations. We can extend the definition of a computable functor to arbitrary classes by allowing the oracles of $\Phi$ and $\Phi_*$ to range over the objects and morphisms of $\mf{C}$. We can then define a functorial notion of reducibility between two classes of structures.

\begin{definition}[\cite{HTMelnikovMillerMontalban}]
Say that a class $\mf{C}$ is \emph{reducible via effective adjoint equivalence} to a class $\mf{D}$ if there exist a subclass $\mf{D}' $ of $\mf{D}$ and computable functors $F\colon \mf{C} \rightarrow \mf{D}'$, $G\colon \mf{D}' \rightarrow \mf{C}$ such that $F \circ G$ and $G \circ F$ are effectively naturally isomorphic to the identity.
\end{definition}

Once again, this is is equivalent to the notion of reducibility by bi-interpretation.

\begin{theorem}[Theorem 1.12 of \cite{HTMelnikovMillerMontalban}]\label{thm:equiv-reducibilities}
A class $\mf{C}$ is reducible via effective bi-interpretability to a class $\mf{D}$ if and only if $\mf{C}$ is reducible via effective adjoint equivalence to the class $\mf{D}$.

Moreover, if $\mf{C}$ is reducible via effective adjoint equivalence to a class $\mf{D}$, the reduction via effective bi-interpretability that one obtains induces a reduction via effective adjoint equivalence that is effectively naturally isomorphic to the original one.
\end{theorem}

\subsection{Universality}

\begin{definition}
A class $\mf{C}$ of structures is \emph{universal} if for each language $\mc{L}$, the class of $\mc{L}$-structures is reducible to $\mf{C}$ via effective bi-interpretability.
\end{definition}

Equivalently, by Theorem \ref{thm:equiv-reducibilities}, a class $\mf{C}$ is universal if and only if for each language $\mc{L}$, the class of $\mc{L}$-structures is reducible to $\mc{F}$ via effective adjoint equivalence.

It follows from \cite{HKSS,MPSS} that each of the following classes is universal: undirected graphs, partial orderings, lattices, and fields, and, after naming finitely many constants, integral domains, commutative semigroups, and 2-step nilpotent groups.

Each of the following classes is not universal: algebraically closed fields, real closed fields, abelian groups, linear orderings, and Boolean algebras. In each case, the computable dimension can only be $1$ or $\omega$.

To define universality among finitely generated structures, we want to restrict to the class of finitely generated $\mc{L}$-structures.

\begin{definition}\label{def:universal}
Let $\mf{C}$ be a class of finitely-generated structures. $\mf{C}$ is \emph{universal among finitely generated structures} if for each language $\mc{L}$, the class of finitely generated $\mc{L}$-structures is reducible to $\mf{C}$ via effective bi-interpretability.
\end{definition}

\section{Groups are Universal for Finitely Generated Structures}\label{sec:groups}

In this section we will prove Theorem \ref{thm:main} in the following strengthened form:

\begin{theorem}\label{thm:main-extra}
Fix a language $\mc{L}$. For each $\mc{L}$-structure $\mc{A}$, we can effectively build a group $G(\mc{A})$ with elements $b,c,d,f_1,f_2 \in G(\mc{A}$) such that $\mc{A}$ is effectively bi-interpretable with $\tilde{G}(\mc{A}) = (G(\mc{A}),b,c,d,f_1,f_2)$. Moreover:
\begin{enumerate}
	\item The formulas of the bi-interpretation do not depend on $\mc{A}$.
	\item The orbit of the tuple $bcdf_1f_2 \in G(\mc{A})$ is definable by a finitary quantifier-free formula.
	\item $\mc{A}$ is finitely generated if and only if $G(\mc{A})$ is.
\end{enumerate}
\end{theorem}

\subsection{Naming Constants}\label{sec:naming-constants}

We will give a brief argument that the class of finitely generated groups is not universal among finitely generated structures, and hence we must name constants.

We will use the fact (see \cite{HTMillerMontalban}) that two structures which are bi-interpretable have the same automorphism group. Every group $G$, $|G| > 2$ has a non-trivial automorphism. On the other hand, there are finitely generated structures which have no non-trivial automorphisms but are not computable. Fix one such structure $\mc{A}$. Since every finitely generated group with no non-trivial automorphisms is computable, no such group can be effectively bi-interpretable with $\mc{A}$. Thus the class of finitely generated groups is not universal. On the other hand, a finitely generated group can become rigid after fixing finitely many constants, for example, its generators.

\subsection{Small Cancellation}

We give a short summary of the definitions and facts we need from small cancellation theory. We refer the interested readers to \cite{LyndonSchupp}.

\begin{definition}
We say a presentation $\langle S \mid R \rangle$ is \emph{symmetrized} if every relation is cyclically reduced and the relation set $R$ is closed under inverse and cyclic permutation.

Let $\langle S \mid R \rangle$ be a symmetrized presentation. We say a word $u \in F(S)$ is a \emph{piece} if there are two $r_1 \neq r_2 \in R$ such that $u$ is an initial subword of both $r_1$ and $r_2$. We also say the presentation satisfies the \emph{$C'(\lambda)$ small cancellation hypothesis} if for every relation $r$ and every piece $u$ with $r = uv$, we have $|u| < \lambda |r|$.
\end{definition}

Furthermore, we shall say a non-symmetrized presentation satisfies the small cancellation hypothesis if it does once we replace the relation set with its symmetrized closure. We shall also say a group is a small cancellation group when it is clear which presentation we are using.

One key lemma we will need for small cancellation groups is the following, which says that every presentation of the trivial word must contain a long common subword with a relator.

\begin{lemma}[Greendlinger's Lemma]
Let $G = \langle S \mid R \rangle$ be a $C'(\lambda)$ small cancellation group with $0 \leq \lambda \leq \frac16$. Let $w$ be a non-trivial freely reduced word representing the trivial element of $G$. Then either $w\in R$, or there are two disjoint subwords $u_1$ and $u_2$ of $w$, such that each of them are a subword of some cyclic permutation $r_i$ of a relation in $R$ or its inverse with $r_i = u_iv_i$ and $|u_i| > (1-3\lambda)|r_i|$.
\end{lemma}

Another lemma we will need is the Torsion Lemma, which says any torsion has to come from a relator.

\begin{lemma}[Torsion Lemma]
Let $G = \langle S \mid R \rangle$ be a $C'(\lambda)$ small cancellation group with $0 \leq \lambda \leq \frac16$. Let $g\in G$ be an element with finite order. Then there is a $r\in R$ such that $r = v^n$ and $g$ is conjugate to a power of $v$ in $G$.
\end{lemma}

We say that a word $w$ is \emph{Dehn-minimal} if it does not contain any subword $v$ that is also a subword of a relator $r = vu$ such that $|v| > |r| / 2$. Greendlinger's lemma implies that, given a $C'(1/6)$ presentation of a group, we can solve the word problem using the following observation: a Dehn-minimal word is equivalent to the identity if and only if it is the trivial word. Given a word $w$, we replace $w$ by equivalent words of shorter length until we have replaced $w$ by a Dehn-minimal word $w'$. Then $w$ is equivalent to the identity if and only if $w'$ is the trivial word. This is Dehn's algorithm.

\subsection{Construction}

Fix a computable language $\mc{L}$. View constants and $0$-ary functions as unary relations, so that we may assume that $\mc{L}$ consists entirely of function and relation symbols with arity $\geq 1$. Given an $\mc{L}$-structure $\mc{A}$, let $G = G(\mc{A})$ be the group with generators $\{ a \}_{a \in A} \cup \{ b, c, d, f_1, f_2 \}$. For elements $g\in G$, we will abuse notation and say $g \in A$ to mean $g$ is one of the generators in the set $\{ a \}_{a \in A}$. We also put the following relations on $G(\mc{A})$:
\begin{itemize}
	\item $f_i^{p_i} = e$, where $p_1,p_2$ are distinct primes greater than $10^{10}$,
	\item $v(f_1,f_2) = v(f_2,f_1) = e$,
	\item $u_b(b,f_1) = u_b(b,f_2) = e$, $u_c(c,f_1) = u_c(c,f_2) = e$, $u_d(d,f_1) = u_d(d,f_2) = e$,
	\item $u_A(a,c) = u_A(a,d) = e$ for $a \in A$,
	\item $w_m(\bar{a},b) = 
	\begin{cases}
a', & \text{if the }m\text{th symbol in }\mc{L}\text{ is a function symbol }f\text{ and }\mc{A} \models f(\bar{a}) = a'\\
e, & \text{if the }m\text{th symbol in }\mc{L}\text{ is a relation symbol }R\text{ and }\mc{A} \models R(\bar{a}),
	
	\end{cases}$ 
\end{itemize}
where
\begin{itemize}
	\item $u_A(x,y) = x y x y^{4} \cdots x y^{999^2} x y^{1000^2}$,
	\item $u_b(x,y) = x y^{1001^2} x y^{1002^2} \cdots x y^{1999^2} x y^{2000^2}$,
	\item $u_c(x,y) = x y^{2001^2} x y^{2002^2} \cdots x y^{2999^2} x y^{3000^2}$,
	\item $u_d(x,y) = x y^{3001^2} x y^{3002^2} \cdots x y^{3999^2} x y^{4000^2}$,
	\item $v(x,y) = x y^{4001^2} x y^{4002^2} \cdots x y^{4999^2} x y^{5000^2}$,
	\item with $n$ the arity of the $m$th symbol in $\mc{L}$, $m \geq 1$,
	\[ w_m(x_1,\ldots,x_n,y) = x_1 \cdots x_n y^{100m+1} x_1 \cdots x_n y^{100m + 2} \cdots x_1 \cdots x_n y^{100m + 100}. \]
\end{itemize}
Then $G$ is a $C'(1/20)$ small cancellation group. To see this, recall that
\[\sum_{i = 1}^n i^2 = \frac{(n)(n+1)(2n+1)}{6}.\]
So, for example, $|u_A(x,y)| = 333,834,500$, so
\[ \frac{|x y^{999^2} x y^{1000^2} x|}{|v(x,y)|} \leq \frac{6}{1000}.\]

The idea of the construction is that set $A \subseteq G(\mc{A})$ is definable from $c$ and $d$, and that the relations and functions on $A$ are definable from $b$ (Lemmas \ref{lem:uu} and \ref{lem:functor-works}). This uses the relators defined using $u_A$ and $w_m$.

We also need to show that the orbits of the elements $b,c,d,f_1,f_2$ are definable by a finitary quantifier-free formula. If we know $f_1$ and $f_2$, we can pick out the set $A$ and the elements $b,c,d$ using Lemma \ref{lem:uu}. To find $f_1$ and $f_2$ (up to a conjugate), we use Lemmas \ref{torsion-implies-f} and \ref{lem:vv}.

We will also need to show that the diagram (i.e., the word problem) of $G(\mc{A})$ is computable from the diagram of $\mc{A}$. This is Lemma \ref{lem:wp}.

For the rest of this section, we often abuse notation and identify an element in $G(\mc{A})$ with a spelling of it in the generating set. We will use usual equality $=$ to denote equality in $G(\mc{A})$, and $\equiv$ to denote equality as words, i.e.\ in the free group generated by the generators.

\begin{lemma}\label{lem:uu}
Fix an $\mc{L}$-structure $\mc{A}$. Then, in $G(\mc{A})$:
\begin{enumerate}
	\item If $u_A(x,c) = u_A(x,d) = e$ then $x \in A$.
	\item If $u_b(x,f_1) = u_b(x,f_2) = e$, but $x^{p_1} \neq e$ and $x^{p_2} \neq e$, then $x = b$.
	\item If $u_c(x,f_1) = u_c(x,f_2) = e$, but $x^{p_1} \neq e$ and $x^{p_2} \neq e$, then $x = c$.
	\item If $u_d(x,f_1) = u_d(x,f_2) = e$, but $x^{p_1} \neq e$ and $x^{p_2} \neq e$, then $x = d$.
\end{enumerate}
\end{lemma}
\begin{proof}
We begin by showing (1). (2), (3), and (4) are similar, but the difference is that (2), (3), and (4) involve $f_1$ and $f_2$, which are elements of finite order. After proving (1), we will describe how to prove (2), highlighting the differences. (3) and (4) are proved in the same way as (2).

Let $x\in G$ be such that $u_A(x,c) = e$. Fix a shortest spelling of $x$ in the generating set $\{ a \}_{a \in A} \cup \{ b, c, d,  f_1, f_2 \}$. By abusing notation, we will write $x$ to mean this fixed spelling of $x$. Write $x \equiv c^k x' c^\ell$ where $k,\ell \in \mathbb{Z}$ and $x'$ begins and ends with a letter that is not $c$. Note that $x'$ does not reduce to the trivial word (if it did, we would have $c^n = e$ for some $n$, which cannot happen by the Torsion lemma).

Then we have
\[ e = u_A(x,c) = x c x c^{4} x c^{9} x c^{16} \cdots x c^{1000^2} = c^k x' c^{1 + k + \ell} x' c^{4 + k + \ell} x' \cdots c^{999^2 + k + \ell} x' c^{1000^2 + \ell} . \]
So
\[x' c^{1 + k + \ell} x' c^{4 + k + \ell} x' \cdots c^{999^2 + k + \ell} x' c^{1000^2 + \ell + k} = e.\tag{$*$}\]
Note that $n^2 + k + \ell$ is $-1$, $0$, or $1$ for at most one value of $n$. There is no cancellation in ($*$) except that $x' x'$ might appear in one place, and $x' x'$ cannot be freely reduced to the trivial word unless $x'$ is already the trivial word. Writing a reduced word in place of $x' x'$, by Greendlinger's Lemma, either the left hand side of ($*$) is in $R$ or there are two subwords satisfying the conclusion of Greendlinger's lemma. If the left hand side of ($*$) is in $R$, then it must be one of the relators $u_A(a,c)$ or its inverse. Thus $x' \equiv a$ and $x \equiv c^k a c^\ell$.

We will argue that $x \equiv c^k a c^\ell$ in the other case as well. There are two subwords $v_1$ and $v_2$ of the left hand side of ($*$) satisfying the conclusion of Greendlinger's Lemma. Since there is at most one $n$ with $n^2 + k + \ell$ equal to $-1$, $0$, or $1$, there is at most one instance of $x'x'$, $x'cx'$, or $x' c^{-1} x'$ as a subword of
\[ x' c^{1 + k + \ell} x' c^{4 + k + \ell} x' \cdots c^{999^2 + k + \ell} x' c^{1000^2 + \ell + k}.\]
We may choose whichever of $v_1$ or $v_2$ (say, without loss of generality, $v \equiv v_1$) which does not intersect the middle of any of these subwords (i.e., it does not involve the free cancellation in $x'x'$, or contain the $c$ in $x'cx'$, or the $c^{-1}$ in $x' c^{-1} x'$). Let $r$ be the relator associated with $v$.
We have a number of possibilities:
\begin{enumerate}[label = (\alph*)]
	\item $v$ is contained in $c^i x' c^j$, where $i$ and $j$ are each $0$, $1$, or $-1$. Then $x'$ is not a shortest spelling, a contradiction.
	\item $v$ fully contains at least two copies of $x'$. Then $v$ also contains $c^2$ or $c^{-2}$, and $x'$ contains a letter which is not $c$, so $r$ must be a cyclic permutation of $u_A(a,c)$ or its inverse for some $a \in A$. Since $x'$ appears twice in $r$, and each cyclic permutation of $r$ is distinct from each other cyclic permutation of $r$, $x'$ is a piece. But each piece appearing in $u_A(a,c)$ contains at most one letter which is not $c$. So $x' \equiv a$, and $x \equiv c^k a c^\ell$.
	\item $v$ fully contains exactly one copy of $x'$. So $v$ is a subword of $x' c^{j^2 + \ell + k} x' c^{(j+1)^2 + \ell + k} x'$ for some $j$. Note that $x'$ cannot contain more than half of $r$, or we would have a shorter spelling of $x$, and so $v$ must contain $c^2$ or $c^{-2}$. As $x'$ contains a letter which is not $c$, $r$ must be a cyclic permutation of $u_A(a,c)$ or its inverse for some $a \in A$.  Also, if $v$ contains the initial or terminal segment of another copy of $x'$, then that segment is a piece and hence has length at most $\frac{1}{20}|r|$. Also, if $c^n$ is a subword of $r$, then $n \leq \frac{1}{20}|r|$. Lastly, $x$ cannot be longer than half of the length of $r$, otherwise it is not a shortest spelling. So $|v| \leq \frac{1}{2}|r| + \frac{4}{20}|r| < \frac{17}{20}|r|$, a contradiction. So this case cannot happen.
	\item None of the above. $v$ is a subword of $x' c^n x'$ for some $n \neq -1,0,1$. $v$ must contain $c^2$ or $c^{-2}$ as a subword and thus $r$ is a cyclic permutation of $u_A(a,c)$ or its inverse for some $a \in A$. Since $x'$ cannot contain more than half of $r$, we can write $v \equiv x_1 c^n x_2$ where $x_1$ is a final segment of $x'$ and $x_2$ is an initial segment of $x'$. We have $|c^n| \leq \frac{1}{20}|r|$ and $|x_1|,|x_2| < \frac12|r|$, so $|x_1|,|x_2| \geq \frac{6}{20}|r|$. Note that in the relator $u_A(a,c) = e$, for any two subwords $y_1$ and $y_2$, there is no cancellation in $y_1y_2$. Since $x_1$ and $x_2$ are subwords of $r$, there cannot be any cancellation in $x_1x_2$, and thus also no cancellation in $x' x'$.

Recall that there was another subword $v_2$ of a relator $r_2$ which we obtained from Greendlinger's lemma. Since $x'x'$ is freely reduced, we may run the above analysis on $v_2$. If $v_2$ is in case (a), then the same arguments work. If $v_2$ is in case (b) or (c), then it contains a copy of $x'$, thus a copy of $x_1$, making it a piece with length $>\frac6{20}|r|$, a contradiction. If $v_2$ is in case (d), then it is a subword of $x' c^m x'$ for some $m \neq n$. As before, we can write $v_2 \equiv x_3 c^m x_4$ with $|x_3|,|x_4| \ge \frac6{20}|r_2|$. Then either $x_2$ is an initial segment of $x_4$ or vice versa; whichever is an initial segment is a piece, yielding a contradiction.
\end{enumerate}

So we have shown that $x$ is of the form $x \equiv c^k a c^\ell$ for some $a \in A$. Now applying the same argument to $u_A(x,d) = e$, we get that $x$ is also of the form $x \equiv d^{k'} a d^{\ell'}$. So $k = \ell = k' = \ell' = 0$ and $x \equiv a$ for some $a\in A$.


\medskip{}

Now we must talk about how to prove (2). Let $x\in G$ be such that $u_b(x,f_1) = e$ and $x^{p_1}\neq e$. Fix a shortest spelling of $x$ in the generating set $\{ a \}_{a \in A} \cup \{ b, c, d,  f_1, f_2 \}$. By abusing notation, we will write $x$ to mean this fixed spelling of $x$. Write $x \equiv f_1^k x' f_1^\ell$ where $k,\ell \in \mathbb{Z}$ and $x'$ begins and ends with a letter that is not $f_1$. Note that since $x^{p_1}\neq e$, $x'$ is not the trivial word.

Then we have
\begin{align*}
e &= u_b(x,f_1) = x f_1^{1001} x f_1^{1002^2} x f_1^{1003^2} x f_1^{1004^2} \cdots x f_1^{2000^2}\\ &= f_1^k x' f_1^{1001^2 + k + \ell} x' f_1^{1002^2 + k + \ell} x' \cdots f_1^{1999^2 + k + \ell} x' f_1^{2000^2 + \ell}.\end{align*}
Let $n_i$ be such that $n_i \equiv (1000+i)^2 + k + \ell \pmod{p_1}$, $|n_i| < p_1 / 2$. Then
\[x' f_1^{n_1} x' f_1^{n_2} x' \cdots f_1^{n_{999}} x' f_1^{n_{1000}} = e.\tag{$*$}\]
Note that $n_i$ is equal to $-1$, $0$, or $1$ for at most one value of $i$. Since each $n_i$ satisfies $|n_i| < p_1/2$, no large subword of a relator $f_1^{p_1} = e$ is a subword of the left hand side of ($*$). Using this fact, we may run the same argument as in (1). In cases (a), (c), and (d) we obtain a contradiction as before. However, in case (b), the relator may be of the form $v(f_2,f_1) = e$, $u_b(b,f_1)=e$, $u_c(c,f_1)=e$, or $u_d(d,f_1)=e$, and so we obtain $x \equiv f_1^k y f_1^\ell$ where $y$ is one of $f_2$, $b$, $c$, or $d$. After running similar argument on $u_b(x,f_2) = e$, we get $x \equiv f_2^{k'} y' f_2^{\ell'}$ where $y$ is one of $f_1$, $b$, $c$, or $d$. So the only possibilities are that $x \equiv f_2 f_1$, $x \equiv f_1 f_2$, $x \equiv b$, $x \equiv c$, or $x \equiv d$. Now by applying Greendlinger's lemma on $u_b(x,f_1) = e$, we get that $x \equiv b$.
\end{proof}

\begin{lemma}\label{lem:functor-works}
Fix an $\mc{L}$-structure $\mc{A}$. Then, in $G(\mc{A})$:
\begin{enumerate}
\item If the $n$th symbol in $\mc{L}$ is a function symbol $f$ and $w_n(\bar{x},b) = y$ with $\bar{x}, y \in A$, then $\mc{A} \models f(\bar{x}) = y$.
\item If the $n$th symbol in $\mc{L}$ is a relation symbol $R$ and $w_n(\bar{x},b) = e$ with $\bar{x}\in A$, then $\mc{A} \models R(\bar{x})$.
\end{enumerate}
\end{lemma}

\begin{proof}
We show this only for (1). The proof of (2) is similar. 

Suppose $w_n(\bar{x},b)y^{-1} = 1$ for some $\bar{x}, y \in A$. Then by Greendlinger's lemma, there is a subword $u$ of $w_n(\bar{x},b)y^{-1}$ such that $r \equiv uv$ for some relator $r$ with $|u| > \frac{17}{20}|r|$. However, as $u$ does not include any of the generators $c$, $d$, $f_1$, or $f_2$, $r$ must be of the form $w_m(\bar{a},b) = a'$ or $w_m(\bar{a},b) = e$.

As a large subword of $r$, $u$ contains segment of the form $a_1\cdots a_kb^{100m+i}a_1\cdots a_k b^{100m+i+1}$. But $u$ is also a subword of $w_n(\bar{x},b)y^{-1}$, which forces $m = n$, and $\bar{a} \equiv \bar{x}$. This means $y = w_m(\bar{a},b) = a$, and so $\mc{A} \models f(\bar{x}) = y$.
\end{proof}

\begin{lemma}\label{torsion-implies-f}
Fix $i = 1$ or $i = 2$. If $x^{p_i} = e$ in $G(\mc{A})$ but $x \neq e$, then $x$ is a conjugate of $f_i^n$ for some $n$ with $p_i \nmid n$.
\end{lemma}

\begin{proof}
By the Torsion Lemma, $x$ is a conjugate of $f_j^n$ for some $j$, $n$. Then $f_j^{n p_i} = e$, and so $p_j \mid n p_i$. So either $i = j$, or $p_j \mid n$. The latter cannot happen, as $x$ is not the identity. So $x = f_i^n$, and since $x$ is not the identity, $p_i \nmid n$.
\end{proof}

\begin{lemma}\label{lem:vv}
Fix an $\mc{L}$-structure $\mc{A}$. Then, in $G(\mc{A})$, if $v(x,y) = v(y,x) = e$, where $x$ and $y$ are conjugates of $f_1^m$ and $f_2^n$ respectively $(p_1 \nmid m$, $p_2 \nmid n$), then $x$ and $y$ are both conjugates of $f_1$ and $f_2$ by the same element.
\end{lemma}

\begin{proof}
Without loss, we may assume that $|m| < p_1/2$. By conjugating both $x$ and $y$, we may also assume that $y \equiv f_2^n$. Assume that $x = zf_1^m z^{-1}$, and fix a shortest spelling of $z$ in the generating set $\{ a \}_{a \in A} \cup \{ b, c, d, f_1, f_2 \}$, so $x \equiv zf_1^m z^{-1}$. By abusing notation, we will write $z$ to mean this fixed spelling of $z$. By conjugating with $f_2$ and/or reducing with $f_1^m$, we may also assume that $z$ does not start with $f_2$ or end with $f_1$. Then we have 
$$ z f_1^m z^{-1} f_2^{4001^2n} z f^m_1 z^{-1} f_2^{4002^2n} \cdots z f_1^m z^{-1} f_2^{5000^2n} = e$$

We may rewrite this to be 
$$ z f_1^m z^{-1} f_2^{n_1} z f^m_1 z^{-1} f_2^{n_2} \cdots z f_1^m z^{-1} f_2^{n_{1000}} = e$$
where $n_i$ is congruent to $(4000+i)^2n$ modulo $p_2$ and $-p_2/2 < n_i < p_2/2$ for all $i$, and there is no reduction.

By Greendlinger's lemma, we get a subword $u$ of the left hand side such that $u$ is also a subword of a relator $r$ with $|u| > \frac{17}{20}|r|$. Suppose $z$ is nontrivial, then $u$ cannot intersect with both some copy of $z$ and some copy of $z^{-1}$, as the letters in each relator are all positive or all negative. However, because $|m| < p_1/2$ and $|n_i| < p_2/2$, $r$ cannot be $f_1^{p_1} = e$ or $f_2^{p_2} = e$. Hence $u$ must intersect with $z$ or $z^{-1}$. Assume that $u$ intersects with $z$, then $u$ must be contained in some $f_2^{k} z f_1^{\ell}$. But any consecutive $f_i$'s in $r$ cannot have length more than $\frac1{20}|r|$, so the intersection of $u$ with $z$ must have length $>\frac12|r|$, so we may get a shorter spelling of $z$, a contradiction. Thus, $z$ must be trivial, and we have 
$$ f_1^m f_2^{n_1} f^m_1 f_2^{n_2} \cdots f_1^m f_2^{n_{1000}} = e$$

Finally arguing by Greendlinger's lemma and noting that the only possibility of $r$ is $v(f_1,f_2) = e$, we get $x \equiv f_1$. By symmetry we also get that $y \equiv f_2$.
\end{proof}

\begin{lemma}\label{lem:orbit}
The orbit of $(b,c,d,f_1,f_2)$ is definable in $G(\mc{A})$ by a finitary quantifier-free formula.
\end{lemma}
\begin{proof}
Let $\varphi(x,y,z,t_1,t_2)$ be the formula which says that:
\begin{itemize}
	\item $t_1^{p_1} = e$ but $t_1 \neq e$,
	\item $t_2^{p_2} = e$ but $t_2 \neq e$,
	\item $v(t_1,t_2) = v(t_2,t_1) = e$,
	\item $u_b(x,t_1) = u_b(x,t_2) = e$ but $x^{p_1} \neq e$ and $x^{p_2} \neq e$,
	\item $u_c(y,t_1) = u_c(y,t_2) = e$ but $y^{p_1} \neq e$ and $y^{p_2} \neq e$,
	\item $u_d(z,t_1) = u_d(z,t_2) = e$ but $z^{p_1} \neq e$ and $z^{p_2} \neq e$.
\end{itemize}
Let $(x,y,z,t_1,t_2)$ satisfy $\varphi$. We claim that, up to conjugation, $(x,y,z,t_1,t_2) = (b,c,d,f_1,f_2)$.

First we claim that $(t_1,t_2)$ is conjugate to $(f_1,f_2)$. By Lemma \ref{torsion-implies-f}, $t_1 = g f_1^m g^{-1}$ with $p_1 \nmid m$ and $t_2 = h f_2^n h^{-1}$ with $p_2 \nmid n$. Then by Lemma \ref{lem:vv}, we get that $(t_1,t_2)$ is conjugate to $(f_1,f_2)$.


We may now assume that $t_1 = f_1$ and $t_2 = f_2$. By Lemma \ref{lem:uu}, we then have $(x,y,z) = (b,c,d)$, which completes the proof. 
\end{proof}

\begin{lemma}\label{lem:wp}
From a presentation of $\mc{A}$, we can compute the word problem of $G(\mc{A})$.
\end{lemma}
\begin{proof}
To determine if a word in the generating set $\{a\}_{a \in A} \cup \{b,c,d,f_1,f_2\}$ represents the identity, we will run Dehn's algorithm, but in the group presentation defined in the construction of $G(\mc{A})$. Note that this group presentation is computable from a presentation of $\mc{A}$. As the group presentation is $C'(1/20)$ small cancellation, Dehn's algorithm, as an infinite abstract procedure, yields correct output.

To use Dehn's algorithm effectively, we have to be able to decide whether, for a given word $w$, $w$ is Dehn-minimal, that is, to decide whether $w$ contains a subword $u$ which is also a subword of a relator $r = uv$ with $|u| > |r|/2$. We claim that we need only check finitely many relators $r$, and that we can effectively compute a list of these relators. First, as $m$ gets larger, the relations having $w_m(\overline{a},b)$ on the left hand side get longer, so there is a finite bound on the values of $m$ we need to check. Second, any subword containing at least half of one of the relations has to contain all the letters used on the left hand side of the relation. So the only relators $r$ which might have a large common subword with $w$ are those whose left hand side contains only letters appearing in $w$. For any given finite set of letters, and bound on $m$, there are only finitely many relators which use those letters, and we can use the diagram of $\mc{A}$ to compute a list of these. Thus Dehn's algorithm is effective.
\end{proof}

\begin{lemma}\label{lem:finite-generation}
$\mc{A}$ is finitely generated if and only if $G(\mc{A})$ is finitely generated.
\end{lemma}
\begin{proof}
If $\mc{A}$ can be finitely generated by $a_1,\ldots,a_n$, then $G(\mc{A})$ can be finitely generated by $a_1,\ldots,a_n,b,c,d,f_1,f_2$.

Now assume $\mc{A}$ is not finitely generated. Suppose, towards a contradiction, that $G(\mc{A})$ is finitely generated by $g_1,\ldots,g_k$. Write each of $g_1,\ldots,g_k$ as a word in $A$ and $b,c,d,f_1$, and $f_2$. Only finitely many letters $a_1,\ldots,a_n$ from $A$ appear in these words, along with possibly $b$, $c$, $d$, $f_1$, and $f_2$. As $\mc{A}$ is not finitely-generated, there must be an $a' \in \mc{A}$ that cannot be written as a term (in the language of $\mc{A}$) in $a_1,\ldots,a_n$. On the other hand, $G(\mc{A})$ is generated by $a_1,\ldots,a_n,b,c,d,f_1,f_2$.

Let $A^* \subseteq A$ be the set generated (in $\mc{A}$) by $a_1,\ldots,a_n \in \mc{A}$. Note that $a' \in A \setminus A^*$. In $G(\mc{A})$, fix a shortest spelling $x$ of $a'$ using the letters $A^* \cup \{b,c,d,f_1,f_2\}$. Now $(a')^{-1} x = e$ in $G(\mc{A})$. By Greendlinger's lemma, there must be a large subword $u$ of some cyclic permutation $r$ of a relator such that $u$ is also a subword of $(a')^{-1}x$.

Because $A^*$ is closed under the application of functions in $\mc{A}$, when we look at the relators of $G(\mc{A})$ we see that if $r'$ is any relator with a large subword $u'$ which uses only the letters $A^* \cup \{b,c,d,f_1,f_2\}$, then $r'$ uses only the letters $A^* \cup \{b,c,d,f_1,f_2\}$. Since $x$ is a shortest spelling using the letters $A^* \cup \{b,c,d,f_1,f_2\}$, it cannot be that $u$ is a subword of $x$. Thus $u$ must contain $(a')^{-1}$.

But $u$ contains only one occurrence of $(a')^{-1}$, so it must be that $r$ is an inverse of a cyclic permutation of a relator the form $w_m(\bar{y},b) = a'$, with $\bar{y} \in A^*$. But then $a'$ can be written as a term in $\bar{y}$ using the $m$th function symbol, a contradiction. We conclude that if $\mc{A}$ is not finitely generated, then $G(\mc{A})$ is not finitely generated.
\end{proof}

We are now ready to put all of our lemmas together to prove Theorem \ref{thm:main-extra}.

\begin{proof}[Proof of Theorem \ref{thm:main-extra}]
Define $\tilde{G}$ by $\tilde{G}(\mc{A}) = (G(\mc{A}),b,c,d,f_1,f_2)$. By Lemma \ref{lem:wp}, $\tilde{G}$ is an effective functor. Given $G = \tilde{G}(\mc{A})$, we can construct a copy $F(G)$ of $\mc{A}$ by taking as its domain the set
\[ \Dom(F(G)) = \{ x \in G \mid u_A(x,c) = u_A(x,d) = e \},\]
and interpreting the $m$th symbol in $\mc{L}$ as either, if it is a relation symbol $R$, the set
\[ R^{F(G)} = \{ \bar{x} \in A \mid w_m(\bar{x},b) = e \},\]
or, if the $m$th symbol is a function $f$, as the set
\[ f^{F(G)}(\bar{x}) = \text{the unique $y$ such that $w_m(\bar{x},b) = y$}.\]
By Lemma \ref{lem:uu} and Lemma \ref{lem:functor-works}, $F(G)$ is isomorphic to $\mc{A}$. It is not hard to see that this is an effective functor. Indeed, it is induced by an interpretation of $\mc{A}$ in $\tilde{G}(\mc{A})$.

We want to show that $F$ and $\tilde{G}$ form a reduction by effective adjoint equivalence between the class of $\mc{L}$-structures and groups with five constants named. To do this, we must show that $F \circ \tilde{G}$ and $\tilde{G} \circ F$ are effectively naturally isomorphic to the identity functors of their respective categories.

Given $\mc{A}$, $F(\tilde{G}(\mc{A}))$ is isomorphic to $\mc{A}$ in an obvious way ($\mc{A}$ injects into $\tilde{G}(\mc{A})$ in a computable way, and $F(\tilde{G}(\mc{A}))$ picks out this subset). If $G \cong \tilde{G}(\mc{A})$, then $F(G)$ picks out an isomorphic copy of $\mc{A}$ which embeds into $G$; thus we can consider $G$ to be generated by $F(G)$ together with $b^G,c^G,d^G,f_1^G,f_2^G \in G$. $\tilde{G}(F(G))$ is generated by the same set $F(G)$, together with $b^{\tilde{G}(F(G))},c^{\tilde{G}(F(G))},d^{\tilde{G}(F(G))},f_1^{\tilde{G}(F(G))},f_2^{\tilde{G}(F(G))} \in \tilde{G}(F(G))$. This induces an obvious isomorphism between $G$ and $\tilde{G}(F(G))$.

Theorem \ref{thm:main-extra} then follows from Lemma \ref{lem:orbit} which implies that the orbit of the tuple $(b,c,d,f_1,f_2) \in G(\mc{A})$ is definable by a finitary quantifier-free formula and Lemma \ref{lem:finite-generation} which implies that $\mc{A}$ is finitely generated if and only if $G(\mc{A})$ is.
\end{proof}

\begin{remark}
$G(\mc{A})$, as a group without the constants named, has an automorphism group that is the semi-direct product of the inner automorphism group and the group of autormophisms $K$ that contains the autormophisms which fix $b,c,d,f_1$, and $f_2$. Furthermore, $K$ is naturally isomorphic to the outer automorphism group of $G(\mc{A})$, as well as the automorphism group of the $\mc{L}$-structure $\mc{A}$, and the automorphism group of $(G(\mc{A}),b,c,d,f_1,f_2)$.
\end{remark}

\section{Classes of Structures That Are Not Universal}\label{sec:fields}

In this section we give some examples of classes of finitely generated structures that are not universal among finitely generated structures. We first observe that there are uncountably many non-isomorphic finitely generated structures, but the class of finitely generated fields, the class of finitely generated commutative rings, and the class of finitely presented groups are all countable. Thus, these classes cannot be universal among finitely generated structures. However, this implies nothing about the computability strength of these structures. Below we give a ``stronger" argument for the non-universality of finitely generated fields, which should work even for other weaker notions of ``universality", by showing they always have low complexity Scott sentences. For instance, finitely generated fields are not universal even if one drops the uniformity from Definition \ref{def:universal}.

Recall that a Scott sentence for a structure $\mc{A}$ is an $\mc{L}_{\omega_1 \omega}$-formula $\varphi$ such that $\mc{A}$ is the only countable model of $\varphi$ up to isomorphism. In \cite{HTHo}, the authors showed that every finitely generated field has a d-$\Sigma^0_2$ Scott sentence. We will use this to argue that finitely generated fields are not universal among finitely generated structures. Recall that Montalb\'an \cite[Lemma 5.3]{MonICM} showed that two structures which are bi-interpretable have Scott sentences of the same complexity, and that we know from \cite{HTHo} that there are finitely generated groups with no d-$\Sigma^0_2$ Scott sentence. Thus there is a finitely generated group which is not bi-interpretable with any finitely generated field.

However, we know from the previous section that while finitely generated groups are not universal, they are after naming constants. We will extend the argument above show that finitely generated fields are not universal even after naming constants.

\begin{proposition}\label{prop:add-const-ss}
Let $\mc{A}$ be a countable structure and $\bar{c} \in \mc{A}$. If $\mc{A}$ has a $\Sigma^0_\alpha$ (respectively $\Pi^0_\alpha$, d-$\Sigma^0_\alpha$) Scott sentence, then so does $(\mc{A},\bar{c})$.
\end{proposition}
\begin{proof}
If $\mc{A}$ has a $\Sigma^0_\alpha$ (respectively d-$\Sigma^0_\alpha$) Scott sentence $\varphi$, then it has a $\Pi^0_{\alpha + 1}$ Scott sentence, and so the orbit of each tuple is definable by a $\Sigma^0_\alpha$ formula (see Theorem of 1.1 \cite{Montalban15}); let $\psi(\bar{x})$ define the orbit of $\bar{c}$. Then $\varphi \wedge \psi(\bar{c})$ is a $\Sigma^0_\alpha$ (respectively d-$\Sigma^0_\alpha$) Scott sentence for $(\mc{A},\bar{c})$.

If $\mc{A}$ has a $\Pi^0_\alpha$ Scott sentence $\varphi$, then the orbit of $\bar{c}$ is defined by a $\Sigma^0_\beta$ formula $\psi(\bar{x})$ for some $\beta < \alpha$. Then $\varphi \wedge \psi(\bar{c})$ is a $\Pi^0_\alpha$ Scott sentence for $(\mc{A},\bar{c})$.
\end{proof}

\begin{theorem}
There is a finitely generated structure which is not bi-interpretable with any finitely generated field, even after naming finitely many constants from the field.
\end{theorem}

This implies that finitely generated fields are not universal among finitely generated structures, even after naming finitely many constants.

\begin{proof}
Let $\mc{A}$ be a finitely generated structure with no d-$\Sigma^0_2$ Scott sentence, and suppose towards a contradiction that it is effectively bi-interpretable with a finitely generated field $F$ possible with finitely many constants $\bar{c}$ named. Now $F$ has a d-$\Sigma^0_2$ Scott sentence, and so by the previous lemma, so does $(F,\bar{c})$. But then $\mc{A}$ also has a d-$\Sigma^0_2$ Scott sentence as $\mc{A}$ is effectively bi-interpretable with $(F,\bar{c})$. This contradiction proves the theorem.
\end{proof}

In some instances, we can also remove constants.

\begin{proposition}
Let $\mc{A}$ be a countable structure and $\bar{c} \in \mc{A}$.
\begin{itemize}
	\item If $(\mc{A},\bar{c})$ has a $\Sigma^0_\alpha$ Scott sentence, then so does $\mc{A}$.
	\item Suppose that the orbit of $\bar{c}$ is defined by a $\Sigma^0_\beta$ formula for some $\beta < \alpha$. If $(\mc{A},\bar{c})$ has a $\Pi^0_\alpha$ (respectively d-$\Sigma^0_\alpha$) Scott sentence, then so does $\mc{A}$.
\end{itemize}
\end{proposition}
\begin{proof}
Suppose that $\varphi(\bar{c})$ is a $\Sigma^0_\alpha$ Scott sentence for $(\mc{A},\bar{c})$. Then $(\exists \bar{x}) \varphi(\bar{x})$ is a $\Sigma^0_\alpha$ Scott sentence for $\mc{A}$.

Suppose that $\varphi(\bar{c})$ is a $\Pi^0_\alpha$ Scott sentence for $(\mc{A},\bar{c})$ and let $\psi$ be a $\Sigma^0_\beta$ definition of the orbit of $\bar{c}$ for some $\beta < \alpha$. Then
\[ (\exists \bar{x}) \psi(\bar{x}) \wedge (\forall \bar{x}) (\psi(\bar{x}) \longrightarrow \varphi(\bar{x}))\] is a $\Pi^0_\alpha$ Scott sentence for $\mc{A}$.

Suppose that $\varphi(\bar{c}) \wedge \gamma(\bar{c})$ is a d-$\Sigma^0_\alpha$ Scott sentence for $(\mc{A},\bar{c})$, with $\varphi$ being $\Sigma^0_\alpha$ and $\gamma$ being $\Pi^0_\alpha$. Let $\psi$ be a $\Sigma^0_\beta$ definition of the orbit of $\bar{c}$ for some $\beta < \alpha$. Then
\[ (\exists \bar{x})[\psi(\bar{x}) \wedge \varphi(\bar{x})] \wedge  (\forall \bar{x}) (\psi(\bar{x}) \longrightarrow \gamma(\bar{x})) \]
is a d-$\Sigma^0_\alpha$ Scott sentence for $\mc{A}$.
\end{proof}

\section{Quasi Scott Sentences}\label{sec:quasi}

\subsection{General Results}

In \cite{HTHo}, the authors proved:

\begin{theorem}\label{thm:scott-eq}
A finitely generated structure $\mc{A}$ has a d-$\Sigma^0_2$ Scott sentence if and only if it does not contain a copy of itself as a proper $\Sigma^0_1$-elementary substructure.
\end{theorem}

If $\mc{A}$ did contain a copy of itself as a $\Sigma^0_1$-elementary substructure, we produced a structure $\mc{A}^* \equiv_2 \mc{A}$ which is not finitely generated, and hence not isomorphic to $\mc{A}$. From this it follows that $\mc{A}$ has no d-$\Sigma^0_2$ Scott sentence, as any d-$\Sigma^0_2$ sentence true of $\mc{A}$ is also true of $\mc{A}^*$.

A Scott sentence is a description of a structure among countable structures; when dealing with finitely generated structures, it is natural to ask whether a structure has a description among finitely generated structures. This is analogous to quasi finite axiomatizations as defined in \cite{Nies} (not to be confused with the different definition of quasi finite axiomatization in \cite{AhlbrandtZiegler} which is unrelated to finitely generated structures.)

\begin{definition}
Let $\mc{A}$ be a finitely-generated structure. A \emph{quasi Scott sentence} for $\mc{A}$ is an $\mc{L}_{\omega_1 \omega}$ sentence $\varphi$ such that $\mc{A}$ is the unique finitely-generated model of $\varphi$.
\end{definition}

Any Scott sentence is automatically a quasi Scott sentence. As every finitely generated structure has a $\Sigma^0_3$ Scott sentence, they all have a $\Sigma^0_3$ quasi Scott sentence. Every finitely generated structure also has a $\Pi^0_3$ quasi Scott sentence (but there are finitely generated structures with no $\Pi^0_3$ Scott sentence):

\begin{proposition}
Every finitely generated structure has a $\Pi^0_3$ quasi Scott sentence.
\end{proposition}
\begin{proof}
Let $\mc{A}$ be a finitely generated structure generated by a tuple $\bar{a}$. Let $p(\bar{x})$ be the atomic type of $\bar{a}$. Then the sentence
\[ \bigwedge_{n}(\forall y_1,\ldots,y_n)(\exists \bar{x}) [y_1,\ldots,y_n \in \la \bar{x} \ra \wedge p(\bar{x})] \]
is a $\Pi^0_3$ quasi Scott sentence for $\mc{A}$. It is clear that this sentence is true of $\mc{A}$. If $\mc{B}$ a finitely generated structure satisfying this sentence, then $\mc{B}$ is generated by a tuple $\bar{b}$; thus there must be some tuple $\bar{b}'$ such that $\bar{b} \in \la \bar{b}' \ra$ and $p(\bar{b}')$. Thus $\bar{b}'$ generates $\mc{B}$, and since $\bar{b}'$ satisfies the atomic type $p$, $\mc{B}$ must be isomorphic to $\mc{A}$.
\end{proof}

Unlike with Scott sentences, we do not have a classification of the structures with a d-$\Sigma^0_2$ quasi Scott sentence. Instead, it seems to be more natural to look at the 2-theory of a structure.

\begin{definition}
Let $\mc{A}$ be a countable structure. The \emph{2-theory} of $\mc{A}$, $\Th_2(\mc{A})$, is the set of $\Sigma^0_2$ and $\Pi^0_2$ sentences true of $\mc{A}$.
\begin{itemize}
	\item If $\mc{A}$ is the only countable model of $\Th_2(\mc{A})$, then we say that $\Th_2(\mc{A})$ is \emph{countably categorical}.
	\item If $\mc{A}$ is finitely generated and the only finitely generated model of $\Th_2(\mc{A})$, then we say that $\Th_2(\mc{A})$ is \emph{quasi-categorical}.
\end{itemize}
\end{definition}
Note that if $\mc{A}$ is finitely generated, every $\Sigma^0_2$ formula in $\Th_2(\mc{A})$ is entailed by a single $\Sigma^0_2$ formula---the one which says that there is a substructure isomorphic to $\mc{A}$.


We have a complete classification of when a structure's 2-theory is quasi-categorical.

\begin{theorem}\label{thm:quasi-cat-eq}
Let $\mc{A}$ be a finitely generated structure. The following are equivalent:
\begin{enumerate}
	\item $\mc{A}$ has a $\Sigma^0_1$-elementary finitely generated substructure $\mc{B} \ncong \mc{A}$ which contains a $\Sigma^0_1$-elementary substructure isomorphic to $\mc{A}$.
	\item $\Th_2(\mc{A})$ is not quasi-categorical.
\end{enumerate}
\end{theorem}
\begin{proof}
Given (1), each structure isomorphic to $\mc{A}$ is contained as a $\Sigma^0_1$-elementary substructure of a structure isomorphic to $\mc{B}$, and vice versa. So we can build a chain
\[ \mc{A}_1 \prec_{\Sigma^0_1} \mc{B}_1 \prec_{\Sigma^0_1} \mc{A}_2 \prec_{\Sigma^0_1} \mc{B}_2 \prec_{\Sigma^0_1} \cdots\]
where each $\mc{A}_i$ is isomorphic to $\mc{A}$, and each $\mc{B}_i$ is isomorphic to $\mc{B}$. Let $\mc{C}$ be the union of this chain. Then $\mc{A} \equiv_2 \mc{C}$ and $\mc{B} \equiv_2 \mc{C}$, so that $\mc{A} \equiv_2 \mc{B}$. Thus $\Th_2(\mc{A})$ is not quasi-categorical.

Suppose that $\Th_2(\mc{A})$ is not quasi-categorical, and let $\mc{B}$ be a non-isomorphic finitely generated structure such that $\mc{B} \equiv_2 \mc{A}$. Let $p$ be the $\Pi^0_1$ type of a generating tuple for $\mc{A}$. Then $\mc{A} \models (\exists \bar{x}) p(\bar{x})$, so $\mc{B} \models (\exists \bar{x}) p(\bar{x})$. Let $\bar{g} \in \mc{B}$ realize $p$. Then $\bar{g}$ generates a copy of $\mc{A}$ which is a $\Sigma^0_1$-elementary substructure of $\mc{B}$. The same argument shows that $\mc{A}$ contains a copy of $\mc{B}$ as a $\Sigma^0_1$-elementary substructure.
\end{proof}

We can use a similar argument to give sufficient, but not necessary, conditions for a structure to have a d-$\Sigma^0_2$ quasi Scott sentence.

\begin{theorem}
Let $\mc{A}$ be a finitely generated structure. Suppose that $\mc{A}$ is contained as a $\Sigma^0_1$-elementary substructure within only countably many (up to isomorphism) finitely generated structures, none of which (other than $\mc{A}$ itself) is a $\Sigma^0_1$-elementary substructure of $\mc{A}$. Then $\mc{A}$ has a d-$\Sigma^0_2$ quasi Scott sentence.
\end{theorem}
\begin{proof}
Let $p$ be the $\Pi^0_1$ type of a generating tuple for $\mc{A}$. Let $\mc{C}$ be the collection of finitely generated structures, not isomorphic to $\mc{A}$, which contain $\mc{A}$ as a $\Sigma^0_1$-elementary substructure. For each $\mc{B} \in \mc{C}$, let $q_{\mc{B}}$ be the $\Pi^0_1$ type of a generating tuple from $\mc{B}$. Then $(\exists \bar{x}) p(\bar{x}) \wedge \bigdoublewedge_{\mc{B} \in \mc{C}} \neg (\exists \bar{x}) q_{\mc{B}}(\bar{x})$ is a quasi Scott sentence for $\mc{A}$.
\end{proof}

\begin{question}
Give a complete classification of the finitely generated structure which have a d-$\Sigma^0_2$ quasi Scott sentence.
\end{question}

\subsection{Quasi Scott Sentences and Bi-Interpretations}

In the next section, we will give some examples of finitely generated structures with various types of quasi Scott sentences. In this section, we prove the results required to use the universality of finitely generated groups to turn these examples into finitely generated groups.

To begin, we need a couple of lemmas about interpretations. The following lemma is well-known, even in the weaker setting of Turing computable embeddings (see \cite{KnightMillerVandenBoom07}). We give a sketch of the proof in the setting of interpretations.

\begin{lemma}\label{lem:reduc}
Let $F$ be a computable functor from a class $\mc{C}$ to a class $\mc{D}$. Let $\varphi$ be a sentence in the language of $\mc{D}$. Then there is a sentence $\psi$ such that for $\mc{A} \in \mc{C}$, $\mc{A} \models \psi$ if and only if $F(\mc{A}) \models \varphi$. Moreover, $\psi$ is the same complexity ($\Sigma^0_\alpha$, $\Pi^0_\alpha$, or d-$\Sigma^0_\alpha$) as $\varphi$.
\end{lemma}
\begin{proof}[Proof sketch]
Each structure $F(\mc{A})$ is interpretable in $\mc{A}$, and the formulas in the interpretations are independent of $\mc{A}$. The formula $\psi$ is obtained by modifying $\varphi$ by relativizing each quantifier to the domain of the interpretation, which is $\Delta^{\comp}_1$-definable, and by replacing each symbol in the language of $F(\mc{A})$ by its $\Delta^{\comp}_1$ definition in the language of $\mc{A}$. Thus the complexity of the sentence is maintained. See \cite[Theorem 5.3.2]{Hodges} for a more detailed proof in the context of interpretations in (finitary) elementary first-order logic.
\end{proof}

This next lemma says that if we have a reduction of the class of $\mc{L}$-structures to the class of $\mc{L}^*$-structures, then we can write down a sentence saying, of a particular $\mc{L}^*$-structure, whether or not it is in the image of this reduction.

\begin{lemma}\label{lem:image}
Suppose that the class of $\mc{L}$-structures is reducible via effective bi-int\-er\-pret\-a\-bil\-it\-y to the class of $\mc{L}^*$-structures. There is a $\Pi^0_2$ $\mc{L}^*$-sentence which defines the image of this reduction.
\end{lemma}
\begin{proof}
Fix $\Delta^{\comp}_1$ formulas for the bi-interpretations. We want to write down a $\Pi^0_2$ $\mc{L}^*$-sentence which says, of a particular $\mc{L}^*$-structure $\mc{A}$, that these $\Delta^{\comp}_1$ formulas define:
\begin{enumerate}
	\item an $\mc{L}$-structure $\mc{B}$ interpreted in $\mc{A}$,
	\item an $\mc{L}^*$-structure $\mc{C}$ interpreted in $\mc{B}$,
	\item an $\mc{L}$-structure $\mc{D}$ interpreted in $\mc{C}$,
	\item an isomorphism between $\mc{A}$ and $\mc{C}$, and
	\item an isomorphism between $\mc{B}$ and $\mc{D}$.
\end{enumerate}
For (1), to say that these $\Delta^{\comp}_1$ formulas define an $\mc{L}$-structure $\mc{B}$ interpreted in $\mc{A}$, we just need to say that the formula defining an equivalence relation is in fact an equivalence relation, that the relation symbols and function symbols are well-defined on equivalence classes, that the function symbols are interpreted as functions, and that the domain is closed under the functions. This can all be expressed as a $\Pi^0_2$ $\mc{L}^*$-sentence.

For (2), to say that the formulas of the bi-interpretation define an $\mc{L}^*$-structure $\mc{C}$ interpreted in $\mc{B}$, we first write down a $\Pi^0_2$ $\mc{L}$-sentence $\chi$ such that $\mc{B} \models \chi$ if and only if the $\Delta^{\comp}_1$ formulas of the bi-interpretation define an $\mc{L}^*$ structure inside of $\mc{B}$. Then using Lemma \ref{lem:reduc} we get an $\mc{L}^*$ sentence $\chi^*$ such that $\mc{A} \models \chi^*$ if and only if $\mc{B} \models \chi$.

We can write down sentences expressing (3), (4), and (5) similarly. Let $\varphi$ be the resulting sentence. Then, for an $\mc{L}^*$-structure  $\mc{A}$, $\mc{A} \models \varphi$ if and only if there is an $\mc{L}$-structure $\mc{B}$ (obtained by (2)) such that $\mc{A}$ and $\mc{B}$ are bi-interpretable using the given $\Delta^{\comp}_1$ formulas.
\end{proof}

Next, we will show that given a finitely generated structure $\mc{A}$, $\mc{A}$ has a d-$\Sigma^0_2$ quasi Scott sentence if and only if $G(\mc{A})$ has a d-$\Sigma^0_2$ Scott sentence, where $G$ is the functor from Section \ref{sec:groups}. We will use more than just the fact that $\mc{A}$ and $G(\mc{A})$ are bi-interpretable. The issue is that the definition of quasi Scott sentences involve finite generation, and a structure which is bi-interpretable with a finitely generated structure is not necessarily finitely generated; all that we can conclude is that it is finitely generated by $\Delta^{\comp}_1$-definable functions, but these functions may not be in the language. However, for the particular functor $G$, we proved in Theorem \ref{thm:main-extra} that $\mc{A}$ is finitely generated if and only if $G(\mc{A})$ is finitely generated.

\begin{proposition}\label{prop:transfer-quasi}
Suppose that $\mc{A}$ is finitely generated. Then $\mc{A}$ has a d-$\Sigma^0_2$ quasi Scott sentence if and only if $G(\mc{A})$ has a d-$\Sigma^0_2$ quasi Scott sentence.
\end{proposition}
\begin{proof}
Suppose that $\mc{A}$ has a d-$\Sigma^0_2$ quasi Scott sentence $\varphi$. By Lemma \ref{lem:image} we can write down a $\Pi^0_2$ sentence $\chi$ which defines the groups $G$ which are isomorphic to $G(\mc{B})$ for some $\mc{L}$-structure $\mc{B}$. By Lemma \ref{lem:reduc} we can then write down a d-$\Sigma^0_2$ sentence $\psi$ in the language of groups such that if $G$ is a group with $G \models \chi \wedge \psi$, then $G = G(\mc{B})$ for some $\mc{L}$-structure $\mc{B} \models \varphi$. We claim that $\chi \wedge \psi$ is a d-$\Sigma^0_2$ quasi Scott sentence for $G(\mc{A})$. Suppose that $G$ is finitely generated and $G \models \chi \wedge \psi$. Then $G = G(\mc{B})$ for some $\mc{L}$-structure $\mc{B} \models \varphi$. By Theorem \ref{thm:main-extra} (3), $\mc{B}$ is finitely generated, and so $\mc{B}$ is isomorphic to $\mc{A}$. Thus $G = G(\mc{B})$ is isomorphic to $G(\mc{A})$.

Suppose that $G = G(\mc{A})$ has a d-$\Sigma^0_2$ quasi Scott sentence $\varphi$. By Lemma \ref{lem:reduc} we can write down a d-$\Sigma^0_2$ sentence $\psi$ which holds of those $\mc{L}$-structures $\mc{B}$ with $G(\mc{B}) \models \varphi$. We claim that $\psi$ is a quasi Scott sentence for $\mc{A}$. Suppose that $\mc{B} \models \psi$ is a finitely generated $\mc{L}$-structure. Then by Theorem \ref{thm:main-extra} (3), $G(\mc{B})$ is a finitely generated model of $\varphi$, and hence isomorphic to $G(\mc{A})$. Since $G(\mc{A})$ and $G(\mc{B})$ are isomorphic, $\mc{A}$ is isomorphic to $\mc{B}$.
\end{proof}

As with Scott sentences in Section \ref{sec:fields}, when we showed that finitely generated fields are not universal, we need to prove that we can add or remove constants from the signature of a structure without changing the complexity of its quasi Scott sentence. Note that Theorem \ref{thm:main-extra} (2) says that the constants in the signature of $G(\mc{A})$ satisfy the hypotheses of this proposition.

\begin{proposition}\label{prop:quasi-const}
Let $\mc{A}$ be a countable structure and $\bar{c} \in \mc{A}$. Suppose that the orbit of $\bar{c}$ is defined by a $\Sigma^0_1$ formula $\psi(\bar{x})$. Then $\mc{A}$ has a $\Sigma^0_2$ (respectively $\Pi^0_2$, d-$\Sigma^0_2$) quasi Scott sentence if and only if $(\mc{A},\bar{c})$ does.
\end{proposition}
\begin{proof}
If $\mc{A}$ has a $\Sigma^0_2$ (respectively $\Pi^0_2$, d-$\Sigma^0_2$) quasi Scott sentence $\varphi$, then $\varphi \wedge \psi(\bar{c})$ is a $\Sigma^0_2$ (respectively $\Pi^0_2$, d-$\Sigma^0_2$) Scott sentence for $(\mc{A},\bar{c})$.

If $(\mc{A},\bar{c})$ has a $\Sigma^0_2$ quasi Scott sentence $\varphi(\bar{c})$ then $(\exists \bar{x}) \varphi(\bar{x})$ is a quasi Scott sentence for $\mc{A}$.

If $(\mc{A},\bar{c})$ has a $\Pi^0_2$ quasi Scott sentence $\varphi(\bar{c})$,
\[ (\exists \bar{x}) \psi(\bar{x}) \wedge (\forall \bar{x}) (\psi(\bar{x}) \longrightarrow \varphi(\bar{x}))\] 
is a quasi Scott sentence for $\mc{A}$.

If $(\mc{A},\bar{c})$ has a d-$\Sigma^0_2$ quasi Scott sentence $\varphi(\bar{c})$. Then
\[ (\exists \bar{x})[\psi(\bar{x}) \wedge \varphi(\bar{x})] \wedge  (\forall \bar{x}) (\psi(\bar{x}) \longrightarrow \gamma(\bar{x})) \]
is a quasi Scott sentence for $\mc{A}$.
\end{proof}

We also have analogues of Proposition \ref{prop:transfer-quasi} and \ref{prop:quasi-const} for $\Th_2(\mc{A})$.

\begin{proposition}\label{prop:transfer-cat}
Suppose that $\mc{A}$ is finitely generated. Then $\Th_2(\mc{A})$ is quasi categorical if and only if $\Th(G(\mc{A}))$ is quasi categorical.
\end{proposition}
\begin{proof}
Suppose that $\Th_2(\mc{A})$ is quasi categorical. By Lemma \ref{lem:image} we can write down a $\Pi^0_2$ sentence $\chi$ which defines the groups $G$ which are isomorphic to $G(\mc{B})$ for some $\mc{L}$-structure $\mc{B}$. Then $\chi \in \Th_2(G(\mc{A}))$. Moreover, for each $\Sigma^0_2$ or $\Pi^0_2$ sentence $\varphi$, by Lemma \ref{lem:reduc} there is a $\Sigma^0_2$ or $\Pi^0_2$ sentence $\varphi^*$ in the language of groups such that if $G$ is a group $G = G(\mc{B})$, then $G \models \varphi^*$ if and only if $\mc{B} \models \varphi$. Thus $\Th_2(G(\mc{A}))$ is quasi categorical.

Suppose that $G = G(\mc{A})$ and $\Th_2(G)$ is quasi categorical. By Lemma \ref{lem:reduc}, for each $\Sigma^0_2$ or $\Pi^0_2$ sentence $\varphi$ in the language of groups, there is a $\Sigma^0_2$ or $\Pi^0_2$ sentence $\varphi^*$ which holds exactly of those $\mc{L}$-structures $\mc{B}$ with $G(\mc{B}) \models \varphi$. Then $\Th_2(\mc{A})$ contains the sentences $\varphi^*$ for each $\varphi \in \Th_2(G)$, and so $\Th_2(\mc{A})$ is quasi categorical.
\end{proof}

\begin{proposition}\label{prop:cat-const}
Let $\mc{A}$ be a finitely generated countable structure and $\bar{c} \in \mc{A}$. Suppose that the orbit of $\bar{c}$ is defined by a $\Sigma^0_1$ formula $\psi(\bar{x})$. Then $\Th_2(\mc{A})$ is quasi categorical if and only if $\Th_2(\mc{A},\bar{c})$ is.
\end{proposition}
\begin{proof}
If $\Th_2(\mc{A})$ is quasi categorical, then $\Th_2(\mc{A},\bar{c}) \supseteq \Th_2(\mc{A}) \cup \{\psi(\bar{c})\}$ is as well.

Suppose that $\Th_2(\mc{A},\bar{c})$ is quasi categorical. As $\mc{A}$ is finitely generated, there is a single $\Sigma^0_2$ formula $\varphi(\bar{c})$ which entails each other $\Sigma^0_2$ formula in $\Th_2(\mc{A},\bar{c})$. Then $\Th_2(\mc{A})$ is quasi categorical as it contains the $\Sigma^0_2$ formula $(\exists \bar{x}) [\psi(\bar{x}) \wedge \varphi(\bar{x})]$ and also contains, for each $\Pi^0_2$ formula $\theta(\bar{c}) \in \Th_2(\mc{A},\bar{c})$, the formula $(\forall \bar{x})[\psi(\bar{x}) \longrightarrow \theta(\bar{x})]$.
\end{proof}

\subsection{Examples}

We know that every finitely generated structure has a $\Sigma^0_3$ Scott sentence, as well as a $\Sigma^0_3$ and a $\Pi^0_3$ quasi Scott sentence. The various possible combinations of Scott sentences and quasi Scott sentences for a structure $\mc{A}$ are as follows, in order from most complicated to describe to simplest to describe:
\begin{enumerate}
	\item $\Th_2(\mc{A})$ is not quasi categorical; then $\mc{A}$ does not have a d-$\Sigma^0_2$ Scott sentence or quasi Scott sentence.
	\item $\Th_2(\mc{A})$ is quasi categorical, $\mc{A}$ does not have a d-$\Sigma^0_2$ quasi Scott sentence (and hence no d-$\Sigma^0_2$ Scott sentence).
	\item $\mc{A}$ has a d-$\Sigma^0_2$ quasi Scott sentence, but no d-$\Sigma^0_2$ Scott sentence.
	\item A d-$\Sigma^0_2$ Scott sentence (which is also a d-$\Sigma^0_2$ Scott sentence).
\end{enumerate}
We know that there are examples of (4), say finitely generated fields, and we know from \cite{HTHo} that there is a finitely generated group which falls into one of (1), (2), or (3), but without proving which. In the remainder of this paper, we will show that there are finitely generated groups falling into (1) and (3). We leave the question of whether there are any structures falling into (2) as an open question.

The fact that there is a finitely generated group with no d-$\Sigma^0_2$ quasi Scott sentence is particularly interesting, as such a group has both a $\Sigma^0_3$ and $\Pi^0_3$ quasi Scott sentence. D. Miller \cite{Miller78} showed that a structure with both a $\Sigma^0_3$ and a $\Pi^0_3$ Scott sentence must have a d-$\Sigma^0_2$ Scott sentence, but our examples shows that the analogous result is not true for quasi Scott sentences.

The strategy for both will be to construct an example in some language, and then to use the universality of finitely generated groups, together with results from the previous section, to obtain a group.

\begin{theorem}\label{thm:first-ex}
There is a finitely generated structure which has no d-$\Sigma^0_2$ Scott sentence, but which does have a d-$\Sigma^0_2$ quasi Scott sentence.
\end{theorem}
\begin{proof}
The language of $\mc{A}$ will consist of unary operators $p$ and $(c_i)_{i \in \mathbb{Z}}$. $\mc{A}$ will be an unrooted tree with $p$ as the parent operator. (It will be unrooted because there will be an infinite sequence of parents.) The domain $A$ of $\mc{A}$ consist of elements $\{ (n,\tau) \mid n \in \omega\text{ and }\tau \in \mathbb{Z}^{<\omega}\}$. Given $\tau \in \mathbb{Z}^{<\omega}$, $\tau \neq \langle \rangle$, define $\tau^-$ to be $\tau$ with the last entry removed. The parent $p(n,\tau)$ of $(n,\tau)$ is $(n,\tau^-)$ if $\tau \neq \la \ra$, and $(n+1,\la \ra)$ otherwise. See Figure \ref{treefigure}. For each $(n,\tau)$, we have the $i$th child operator $c_i(n,\tau) = (n,\tau \concat i)$. Note that $\mc{A}$ is generated by $(0,\la \ra)$. Indeed, $(n,\la i_1,\ldots,i_\ell \ra) = c_{i_\ell} \circ \cdots \circ c_{i_1} \circ p^n (0,\la \ra)$.

\begin{sidewaysfigure}[pt]
\[ \xymatrix @C=0pc @R=1pc{
&&&&&\iddots \ar@{-}[ddl]\\\\
&&&&(2,\la \ra)\ar@{-}[ddddl]\ar@{-}[dddr]\ar@{-}[dddrr]\ar@{-}[dddrrr]\ar@{-}[dddrrrr]\ar@{-}@/^/[dddrrrrr]\\\\\\
&&&&&\cdots & (2,\la -1 \ra)\ar@{-}[d] & (2,\la 0 \ra)\ar@{-}[d] & (2,\la 1 \ra)\ar@{-}[d] & \cdots
\\
&&&(1,\la \ra)\ar@{-}[dddddl]\ar@{-}[dddr]\ar@{-}[dddrr]\ar@{-}[dddrrrrr]\ar@{-}@/^/[dddrrrrrrrr]\ar@{-}@/^1pc/[dddrrrrrrrrr]&&&\vdots&\vdots&\vdots
\\\\\\
&&&&\cdots & (1,\la -1 \ra)\ar@{-}[d] &&& (1,\la 0 \ra)\ar@{-}[d]\ar@{-}[dl]\ar@{-}@/_1pc/[dll]\ar@{-}[dr]\ar@{-}@/^1pc/[drr] &&& (1,\la 1 \ra)\ar@{-}[d] & \cdots
\\
&&&&&\vdots & \cdots & (1,\la 0,-1 \ra)\ar@{-}[d] & (1,\la 0,0 \ra)\ar@{-}[d] & (1,\la 0,1 \ra)\ar@{-}[d] & \cdots & \vdots
\\
&&(0,\la \ra)\ar@{-}[dddll]\ar@{-}[dddl]\ar@{-}[dddrrrr]\ar@{-}@/^1pc/[dddrrrrrrrrr]\ar@{-}@/^2pc/[dddrrrrrrrrrr]&& & & & \vdots & \vdots & \vdots &&&
 \\\\\\
\cdots & (0,\la -1 \ra)\ar@{-}[d]\ar@{-}[dr]\ar@{-}[dl] &&&&& (0,\la 0 \ra)\ar@{-}[d]\ar@{-}[dll]\ar@{-}@/_1pc/[dlll]\ar@{-}[drr]\ar@{-}@/^1pc/[drrr] &&&&& (0,\la 1 \ra)\ar@{-}[d]\ar@{-}[dr]\ar@{-}[dl] & \cdots
\\
\cdots & (0,\la -1,0 \ra)\ar@{-}[d] & \cdots & \cdots & (0,\la 0,-1 \ra)\ar@{-}[d] && (0,\la 0,0 \ra)\ar@{-}[d]\ar@{-}[dr]\ar@{-}[dl] && (0,\la 0,1 \ra)\ar@{-}[d] & \cdots & \cdots & (0,\la 1,0 \ra)\ar@{-}[d] & \cdots
\\
& \vdots & & & \vdots & \cdots & (0,\la 0,0,0 \ra)\ar@{-}[d] & \cdots & \vdots &&& \vdots
\\
&&&&&& \vdots
} \]
\caption{The tree $(A,p)$.}\label{treefigure}
\end{sidewaysfigure}

To see that $\mc{A}$ has no d-$\Sigma^0_2$ Scott sentence, we will use Theorem \ref{thm:scott-eq}, showing that $\mc{A}$ contains a copy of itself as a proper $\Sigma_1$-elementary substructure. Let $\mc{B}$ be the substructure generated by $(1,\la \ra)$. It is easy to see that $(0,\la \ra) \notin \mc{B}$, so that $\mc{B}$ is a proper substructure of $\mc{A}$. To see that $\mc{B} \prec_{\Sigma_1} \mc{A}$, let $\bar{b}$ be a tuple of elements from $\mc{B}$ and let $\varphi(\bar{x},\bar{y})$ be a quantifier-free formula with $\mc{A} \models (\exists \bar{x}) \varphi(\bar{x},\bar{b})$. We must show that $\mc{B} \models (\exists \bar{x}) \varphi(\bar{x},\bar{b})$. Let $\bar{a} \in \mc{A}$ be such that $\mc{A} \models \varphi(\bar{a},\bar{b})$. Let $k$ be sufficiently large that $\varphi$ involves only the symbols $p$ and $(c_i)_{|i| < k}$, and each element of $\bar{a}$ and $\bar{b}$ is of the form $(n,\tau)$ with each entry of $\tau$ smaller than $k$. Let $\bar{a}' \in \mc{B}$ be obtained from $\bar{a}$ by replacing each element of $\bar{a}'$ of the form $(0,\tau)$ by $(1,k \concat \tau)$. Then the same relations from $p$ and $(c_i)_{|i| < k}$ hold between $\bar{a},\bar{b}$ and $\bar{a}',\bar{b}$. So $\mc{B} \models \varphi(\bar{a}',\bar{b})$ as desired.

Let $\varphi$ be the d-$\Sigma^0_2$ sentence which is the conjunct of the sentences which say:
\begin{itemize}
	\item there is a $\Sigma^0_1$-elementary substructure isomorphic to $\mc{A}$,
	\item for every two elements, one generates the other,
	\item $(\forall x)[ \la x \ra \cong \mc{A} \vee \bigdoublevee_{i \in \mathbb{Z}} (\exists y) [c_i(y) = x] ]$,
\end{itemize}
These are $\Sigma^0_2$, $\Pi^0_2$, and $\Pi^0_2$ respectively. The third sentence says that each element which is not the image of an element under some $c_i$ generates a copy of $\mc{A}$. It is not hard to see that $\mc{A} \models \varphi$.

Let $\mc{B}$ be a finitely generated model of this sentence $\varphi$; using the second conjunct, $\mc{B}$ is in fact generated by a single element $b$. We claim that we may assume that the generator $b$ is not the image of any element under $c_i$; then by the third conjunct, $b$ generates a copy of $\mc{A}$ and we are done. If $b = c_i(x)$, then replace $b$ by $x = p(b)$, which is still a generator of $\mc{B}$. We claim that this process will end at some point. Using the first conjunct, we may assume that $\mc{A} \preceq_{\Sigma_1} \mc{B}$. Let $a = (0,\la \ra) \in \mc{A}$. Then $a$ is generated by $b$, say by a term $t$. In $\mc{A}$, $a \neq c_i(x)$ for any $x$; thus the same is true in $\mc{B}$. So $t = p \circ \cdots$. Moreover, for any element $x \in \mc{A}$, $p(c_i(x)) = x$, and so the same is true in $\mc{B}$. Thus $t$ is equivalent to $p^n$ for some $n$, and $a = p^n(b)$. So the process described above can go at most $n$ steps.
\end{proof}

\begin{corollary}
There is a finitely generated group which has no d-$\Sigma^0_2$ Scott sentence, but which does have a d-$\Sigma^0_2$ quasi Scott sentence.
\end{corollary}
\begin{proof}
Let $\mc{A}$ be the structure from the previous theorem. Let $G$ be the finitely generated group and $\bar{c} \in G$ the 5-tuple of elements with $(G,\bar{c}) = \tilde{G}(\mc{A})$. Since $\mc{A}$ and $(G,\bar{c})$ are effectively bi-interpretable, $(G,\bar{c})$ has no d-$\Sigma^0_2$ Scott sentence. By Propositions \ref{prop:add-const-ss}, $G$ has no d-$\Sigma^0_2$ Scott sentence.

By Proposition \ref{prop:transfer-quasi}, $(G,\bar{c})$ has a d-$\Sigma^0_2$ quasi Scott sentence. By Theorem \ref{thm:main-extra} (2), the orbit of $\bar{c}$ is $\Sigma^0_1$-definable. By Proposition \ref{prop:quasi-const}, $G$ has a d-$\Sigma^0_2$ quasi Scott sentence.
\end{proof}

\begin{theorem}
There is a finitely generated structure $\mc{A}$ such that $\Th_2(\mc{A})$ is not quasi categorical.
\end{theorem}
\begin{proof}
The structure $\mc{A}$ will be the same as the structure of Theorem \ref{thm:first-ex}---consisting of unary operators $p$ and $(c_i)_{i \in \mathbb{Z}}$---but with the addition of a new unary relation $P$. The domain $A$ of $\mc{A}$ again consist of elements $\{ (n,\tau) \mid n \in \omega\text{ and }\tau \in \mathbb{Z}^{<\omega}\}$. As before, the parent $p(n,\tau)$ of $(n,\tau)$ is $(n,\tau^-)$ if $\tau \neq \la \ra$, and $(n+1,\la \ra)$ otherwise, and $c_i(n,\tau) = (n,\tau \concat i)$. We set $P(n,\tau)$ if $n + |\tau|$ is even. Then $P$ holds of the elements at every second level of the tree.

Let $\mc{B}$ be the structure generated by $(1,\la \ra)$. Then $\mc{A}$ and $\mc{B}$ are not isomorphic. Indeed, $\mc{A}$ has a generator $(0,\la \ra)$ on which $P$ holds and is not an image of any of the $c_i$'s, but $\mc{B}$ does not have such a generator.

We will show that $\mc{A} \preceq_{\Sigma_1} \mc{B}$ and $\mc{B} \preceq_{\Sigma_1} \mc{A}$, from which it will follow by Theorem \ref{thm:quasi-cat-eq} that $\Th_2(\mc{A})$ is not quasi categorical. To see that $\mc{A} \preceq_{\Sigma_1} \mc{B}$, we use the same argument as in Theorem \ref{thm:first-ex}, noting that $P(0,\tau) \Longleftrightarrow P(1,k \concat \tau)$. The argument that $\mc{B} \preceq_{\Sigma_1} \mc{A}$ is similar.
\end{proof}

\begin{corollary}
There is a finitely generated group $G$ such that $\Th_2(G)$ is not quasi-categorical.
\end{corollary}
\begin{proof}
Let $\mc{A}$ be the structure from the previous theorem. Let $G$ be the finitely generated group and $\bar{c} \in G$ the 5-tuple of elements with $(G,\bar{c}) = \tilde{G}(\mc{A})$. By Proposition \ref{prop:transfer-cat}, $\Th_2(G,\bar{c})$ is not quasi categorical. By Theorem \ref{thm:main-extra} (2), the orbit of $\bar{c}$ is $\Sigma^0_1$-definable. By Proposition \ref{prop:cat-const}, $\Th_2(G)$ is not quasi categorical.
\end{proof}

\begin{question}
Is there a finitely generated structure $\mc{A}$ such that $\Th_2(\mc{A})$ is quasi-categorical, but $\mc{A}$ has no d-$\Sigma^0_2$ quasi Scott sentence?
\end{question}

\bibliography{References}

\begin{thebibliography}{HTMMM17}

\bibitem[AZ86]{AhlbrandtZiegler}
G.~Ahlbrandt and M.~Ziegler.
\newblock Quasi-finitely axiomatizable totally categorical theories.
\newblock {\em Ann. Pure Appl. Logic}, 30(1):63--82, 1986.
\newblock Stability in model theory (Trento, 1984).

\bibitem[HKSS02]{HKSS}
D.~R. Hirschfeldt, B.~Khoussainov, R.~A. Shore, and A.~M. Slinko.
\newblock Degree spectra and computable dimensions in algebraic structures.
\newblock {\em Ann. Pure Appl. Logic}, 115(1-3):71--113, 2002.

\bibitem[Hod93]{Hodges}
W.~Hodges.
\newblock {\em Model theory}, volume~42 of {\em Encyclopedia of Mathematics and
  its Applications}.
\newblock Cambridge University Press, Cambridge, 1993.

\bibitem[HTH]{HTHo}
M.~Harrison-Trainor and M.-C. Ho.
\newblock Scott sentences of finitely generated algebraic structures.
\newblock Preprint.

\bibitem[HTMM]{HTMillerMontalban}
M.~Harrison-Trainor, R.~Miller, and A.~Montalb\'an.
\newblock Borel functors and infinitary interpretations.
\newblock Preprint.

\bibitem[HTMMM17]{HTMelnikovMillerMontalban}
M.~Harrison-Trainor, A.~Melnikov, R.~Miller, and A.~Montalb\'an.
\newblock Computable functors and effective interpretability.
\newblock {\em J. Symb. Log.}, 82(1):77--97, 2017.

\bibitem[KMVB07]{KnightMillerVandenBoom07}
J.~F. Knight, S.~Miller, and M.~Vanden~Boom.
\newblock Turing computable embeddings.
\newblock {\em J. Symbolic Logic}, 72(3):901--918, 2007.

\bibitem[LS01]{LyndonSchupp}
Roger~C. Lyndon and Paul~E. Schupp.
\newblock {\em Combinatorial group theory}.
\newblock Classics in Mathematics. Springer-Verlag, Berlin, 2001.
\newblock Reprint of the 1977 edition.

\bibitem[Mar02]{Marker02}
D.~Marker.
\newblock {\em Model Theory : An Introduction}.
\newblock Graduate Texts in Mathematics. Springer, 2002.

\bibitem[Mil78]{Miller78}
D.~E. Miller.
\newblock The invariant $\pi^0_\alpha$ separation principle.
\newblock {\em Trans. Amer. Math. Soc.}, 242:185--204, 1978.

\bibitem[Mon]{MonICM}
A.~Montalb\'an.
\newblock Computability theoretic classifications for classes of structures.
\newblock {\em Proccedings of the ICM 2014}.
\newblock To appear.

\bibitem[Mon13]{MonFixed}
A.~Montalb\'an.
\newblock A fixed point for the jump operator on structures.
\newblock {\em Journal of Symbolic Logic}, 78(2):425--438, 2013.

\bibitem[Mon15]{Montalban15}
Antonio Montalb\'an.
\newblock A robuster {S}cott rank.
\newblock {\em Proc. Amer. Math. Soc.}, 143(12):5427--5436, 2015.

\bibitem[MPSS]{MPSS}
R.~Miller, B.~Poonen, H.~Schoutens, and A.~Shlapentokh.
\newblock A computable functor from graphs to fields.
\newblock Preprint.

\bibitem[Nie03]{Nies}
A.~Nies.
\newblock Separating classes of groups by first-order sentences.
\newblock {\em Internat. J. Algebra Comput.}, 13(3):287--302, 2003.

\end{thebibliography}
\bibliographystyle{alpha}

\end{document}